\documentclass[10pt]{amsart}

\usepackage{graphicx}
\usepackage{setspace}
\usepackage{marginnote}
\usepackage{amsrefs}
\usepackage[top=3.5cm, bottom=2.5cm, outer=2.5cm, inner=2.5cm]{geometry}
\usepackage{hyperref}
\hypersetup{linkcolor=red,
citecolor=black
}

\newtheorem{thm}{Theorem}[section]
\newtheorem{lem}[thm]{Lemma}

\theoremstyle{definition}
\newtheorem{df}[thm]{Definition}
\newtheorem{ex}[thm]{Example}

\theoremstyle{remark}

\numberwithin{equation}{section}
\newcommand{\bnd}{\mathbf{bnd}}
\newcommand{\p}{\mathbf{P}}
\newcommand{\h}{\sum_g\mu(g)\log{\mu(g)}}

\newcommand{\x}{\boldsymbol{x}}
\newcommand{\e}{\boldsymbol{E}}

\newcommand{\g}{\mathcal{G}}
\begin{document}

\title{Asymptotic entropy of transformed random walks}

\author{Behrang Forghani}
\address{Department of Mathematics, University of Ottawa, Canada}
\email{bforg048@uottawa.ca}
\thanks{This work is supported by NSERC (natural sciences and engineering research council) and  CRC (Canada research chairs program)}






\begin{abstract}
We consider general transformations of  random walks on groups determined by  Markov stopping times and prove that the asymptotic entropy (resp., rate of escape) of the
 transformed random walks is equal to the asymptotic entropy (resp., rate of escape) of the original  random walk multiplied by the expectation of the corresponding  stopping time. This is an analogue of the well-known Abramov's formula from ergodic theory, its  particular cases were established
 earlier by Kaimanovich [1983] and Hartman, Lima, Tamuz [2014].
\end{abstract}

\maketitle

\section*{Introduction}
The notion of entropy of a countable probability space was introduced by Shannon in 1948.
He used it to define the asymptotic entropy (entropy rate) in order to  quantify the amount of information for a stationary stochastic process \cite{Sh48}.
Later in the mid 1950's  Kolomogorov  developed the notion of entropy of a measure preserving dynamical system \cite{Kol58}, and his work was completed  by Sinai \cite{Sin59}.
But, it was only  in 1972, that Avez defined the asymptotic entropy of a random walk on a group \cite{A72}.
Despite a formal similarity, the contexts of these definitions are different, and so far
there is no common approach which would unify them.

\medskip
The asymptotic entropy is an important quantity which describes the behavior of a random walk at infinity.
For instance, the triviality of the Poisson boundary of a random walk is equivalent to vanishing of the asymptotic entropy \cite{ A74,De80,KV83}.

\medskip

There are various formulas for the asymptotic entropy of a random walk on a group:
\begin{enumerate}
\item [\upshape(i)] In terms of the entropy of convolution powers \cite{KV83}, see equality (\ref{def asy}) below,
\item [\upshape(ii)] Shannon's formula \cite{De80,KV83}, see Theorem~\ref{shannon} below,
\item [\upshape(iii)]\label{iii} As the average of  {\it Kullback--Liebler deviations} between the harmonic measure and its translates  \cite{K83,KV83}, see equality (\ref{mubnd}) below,
\item [\upshape(iv)] As the exponential growth rate of the Radon--Nikodym derivatives of  the translates of the harmonic measure along sample paths of the random walk \cite{KV83}, see equality  (\ref{harmonic}) below.
\end{enumerate}

In the last two formulas, the asymptotic entropy is expressed in terms of the Poisson boundary of the random walk,
which suggests considering a possible relationship between asymptotic entropies for random walks on the same group which share
a common Poisson boundary.

Earlier this relationship was studied in two particular situations:
\begin{enumerate}
\item [(j)] convex combinations of  convolutions of a given probability measure  \cite{K83},
\item [(jj)] the induced random walk  on a recurrent subgroup \cite{Fu70, K91, K90, Y13}.
\end{enumerate}

In case (j), the asymptotic entropy can be obtained by a direct calculation based  on formula \upshape(iii) \cite{K83}.

 In case (jj), Furstenberg \cite{Fu70} introduced induced random walk on a recurrent subgroup and proved  that its Poisson boundary is the same as the original Poisson boundary.
Kaimanovich \cite{K90} used a similar model to study harmonic functions on a Riemannian manifold and  to compare the asymptotic entropies in this context. Although his setup was somewhat different, by the same approach one can also find the asymptotic entropy of the induced random walk  on a recurrent subgroup
\cite{K91}. Recently, Hartman, Lima and Tamuz  \cite{Y13} calculated the asymptotic entropy of the random walk induced on a finite index subgroup  in an alternative way by using formula \upshape(iii) (although, apparently,  they were not aware of \cite{K90} and \cite{K91}).

\medskip

The probability measures arising in the above situations are examples of transformations of probability measures which do not change the Poisson boundary. Finding all probability measures with the same Poisson boundary is important for understanding the structure of a group $G$. For instance, Furstenberg \cite{Fu70} proved that the Poisson boundary of $SL(n,\Bbb R)$ is the same as the Poisson boundary of a lattice of  $SL(n,\Bbb R)$ endowed with an appropriate measure. By using this fact, he concluded that the lattices of  $SL(n,\Bbb R)$ for $n\geq3$ are different from discrete subgroups of $SL(2,\Bbb R)$.
\medskip

Despite the existence of several examples, currently there is no general way  to find all probability measures
on a group $G$ which have the same Poisson
boundary as a given probability measure $\mu$.  However,
Kaimanovich and the author \cite{BK2013} proposed  a method  to construct many probability measures whose Poisson boundary coincides with that of $\mu$.
This method consists in applying  a Markov stopping time to the original random walk (or, to its randomization) and all currently known examples can be obtained by this method.
\medskip

The purpose of this article is to show how these transformations affect the asymptotic entropy.
We will show that the asymptotic entropy $h'$ of the transformed random walk is the result of rescaling the asymptotic entropy $h$ of the original random walk by the expectation $\boldsymbol{\tau}$ of the stopping time  (Theorem~\ref{markov extension}):
\begin{equation}\label{abra}
h'=\boldsymbol{\tau} h.
\end{equation}
The aforementioned examples (j) and (jj) are contained in this result as  particular cases.

\medskip

Equation~(\ref{abra}), the rescaling  of the asymptotic entropy under a  ``time change'', is analogous to Abramov's formula
 \cite{Ab59} for the entropy  of induced dynamical systems (Theorem~\ref{abramov}). However, as we have already pointed out,
 we are not aware of any common context which would unify these two formulas.

Our proof consists of three steps. Firstly, by using the martingale theory, we  prove that  finiteness of the entropy   of a probability measure is preserved after applying a  Markov stopping time with a finite expectation (Lemma~\ref{finite}).  Secondly, by taking into account
formula \upshape(iv) for the asymptotic entropy, we will establish the main result. And finally, applying the same method, we will prove that this result holds for randomized Markov stopping times as well (Theorem~\ref{markov extension}).

We would like to emphasize that in  our general setup finiteness of the expectation   $\boldsymbol\tau$ is not related to finiteness of any associated space (which can already be observed in the  case of convolution powers, see above (j) ). On the other hand, the technique used by Hartman, Lima and Tamuz \cite{Y13} crucially depends
on the fact that for the induced random walks on  recurrent subgroups  $\boldsymbol\tau<\infty$ if and only if the subgroup has finite index, in combination
with a  number of properties of finite state Markov chains (formulated in the Appendix to \cite{Y13}).

\medskip

The rate of escape  is another quantity which describes behavior of a random walk at infinity.
There are some interrelations between the rate of escape and asymptotic entropy  \cite{V85, KL07}.
We will show that  the rate of escape of a transformed random walk under a randomized Markov stopping time is also transformed according to formula  (\ref{abra}).

\medskip

\subsection*{Acknowledgment}
I am very grateful to my supervisor, Vadim Kaimanovich,
whose support  and patience have enabled me
to develop an understanding of the subject.
\section{Preliminaries}
In this section, we will recall the basic definitions related to random walks on a group, the associated Poisson boundary and  transformations of a random walk via a Markov stopping time.
\subsection{Random walks on groups}
Throughout this paper, we assume that $G$ is a countable group with the identity element $e$
endowed with a probability measure $\mu$.  The random walk $(G,\mu)$ is the Markov chain on $G$ with the
transition probabilities
$$
p(g_1,g_2)=\mu(g^{-1}_1g_2).
$$
The {\it space of increments} of the random walk $(G,\mu)$ is denoted by $(G^{\mathbb N},\bigotimes\limits_{1}^{\infty}\mu)$ ($\bigotimes$ denotes the product measure), equivalenty,
it is the set of sequences $(h_n)_{n\geq 1}$ such that $h_n$'s are  independent and identically $\mu$-distributed.
The space of sample paths of the random walk $(G,\mu)$ is the probability
space $(G^{\Bbb Z+},\p)$,  where $G^{\Bbb Z+}=\{e\}\times G^{\Bbb N}$ (we only consider  random walks issued from the identity),
and the probability measure $\p$  is the image of the measure
$\bigotimes_{1}^{\infty}\mu$ under the map
$$
(h_1,h_2,\cdots)\mapsto \x=(e,x_1,x_2,\cdots),
$$
where $x_n=h_1\cdots h_n$ is the position of the random walk $(G,\mu)$ at time $n$.

\subsection{Poisson boundary}
Let $T$ be the time shift on the space of sample paths, i.e., $T(x_n)=(x_{n+1})$. Let
$\mathcal{A}_T$ be the invariant $\sigma$-algebra of the time shift $T$.
In other words,
$$
\mathcal{A}_T=\{A\ \mbox{measurable } :\ A=T^{-1}A \}.
$$
Rokhlin's correspondence theorem  \cite{Ro52} implies that  there exists
a quotient map $\bnd$ from $(G^{\Bbb Z+},\p)$ onto a Lebesgue space $(\Gamma,\nu)$
such that the pre-image of the $\sigma$-algebra of $(\Gamma,\nu)$ under the map $\bnd$ coincides (mod~0) with $\mathcal{A}_T$. The image of the measure  $\p$ under the map $\bnd$  is called {\it harmonic measure}.
The Lebesgue space $(\Gamma,\nu)$ is called the {\it Poisson boundary} of the random walk $(G,\mu)$.

\medskip

The group $G$ acts naturally on the Poisson boundary,  and $\nu$ is a $\mu$-stationary measure, that is,
$$
\mu*\nu=\sum_g\mu(g)g\nu=\nu.
$$

Let $\mu'$ be another probability measure on group $G$ whose Poisson boundary is $(\Gamma',\nu')$.
We say Poisson boundaries $(\Gamma,\nu)$ and $(\Gamma',\nu')$  are  the same, whenever these boundaries are isomorphic as two
measure $G$-spaces. More precisely,
there exists a bijective measurable map $\phi:(\Gamma,\nu)\to (\Gamma',\nu')$ such that
$\nu'$ is the image of the probability measure $\nu$ and $g\phi(\gamma)=\phi(g\gamma)$
for every $g\in G$, and almost every $\gamma\in\Gamma$.
\subsection{Transformed random walks}

Let $\tau:(G^{\Bbb Z+},\p)\mapsto \Bbb N$ be a Markov stopping time on the space of sample paths, i.e., $\tau$
is a measurable map such that for every natural number $s$ the set $\tau^{-1}\{s\}$
belongs to $\mathcal{A}_0^s$, the $\sigma$-algebra generated by the position of the random walk between time 0 and $s$ of the random walk $(G,\mu)$.
[We only consider  Markov stopping times which are finite almost everywhere!].

\medskip

Random walks on groups are time and space homogenous, which allows us to iterate the Markov stopping time $\tau$ and produce a new random walk \cite{BK2013}. More precisely,  let $\tau_1=\tau$, and define by induction
\begin{equation}\label{iteration}
\tau_{n+1}=\tau_n+\tau(U^{\tau}),
\end{equation}
where
\begin{equation}\label{induced}
U(\x)=(x_1^{-1}x_{n+1})_{n\geq0}
\end{equation}
 is the transformation of the path space induced by the time shift in the space of increments.

\medskip
Then, $(x_{\tau_n})$ is a sample path of the random walk $(G,\mu_{\tau})$, where
$\mu_{\tau}$ is  the distribution of  $x_{\tau}$, i.e.,
$$
\mu_{\tau}(g)=\p\{\x\ :\ x_{\tau}=g\}.
$$
Obviously, each $\tau_n$ is also a Markov stopping time, and, moreover,
the distribution of $x_{\tau_n}$ is the $n$-fold convolution of $\mu_{\tau}$, i.e, $\mu_{\tau_{n}}=(\mu_{\tau})^{*n}$.

\medskip

Straightforward examples of this construction are
\begin{ex}\label{convolution}
Let $k$ be a positive integer, and let  $\tau$ be the constant function $k$. Then $\mu_{\tau}$ is the $k$-fold convolution of measure $\mu$,
i.e., $\mu^{*k}$.
\end{ex}

\begin{ex}
Let $\tau_1$ and $\tau_2$ be two Markov stopping times for the random walk $(G,\mu)$. Then,
$\tau=\tau_1+\tau_2(U^{\tau_1})$ is a Markov stopping time,
and $\mu_{\tau}=\mu_{\tau_1}*\mu_{\tau_2}$.
\end{ex}

\begin{ex}\label{wil}
Let $B$ be a subset of $G$ with $\mu(B)>0$.  For a sample path $\x=(x_n)$, define
$$
\tau(\boldsymbol{x})=\min\{i\geq1\ :\ h_i\in B\},
$$
be the minimal time $i$ such that the increment $h_i$ belongs to $B$.
Then $\tau$ is a  Markov stopping time. If $\mu(B)=1$, then trivially $\mu_{\tau}=\mu$, otherwise
$$
\mu_{\tau}=\beta+\sum_{i=1}^{\infty}\alpha^{*i}*\beta,
$$
where $\beta$ is the restriction of $\mu$ into $B$, i.e.,  $\beta(A)=\mu(A\cap B)$ for a subset $A$ of $G$, and $\alpha=\mu-\beta$.
\end{ex}

\begin{ex}\label{recurrent}
Let $H$ be a  recurrent subset of $G$, i.e.,  $\p\{\x:\ x_n\in H \mbox{\ for\ some } n\}=1$.
Define $$\tau(\boldsymbol{x})=\min\{i\geq1\ x_i\in H\}.$$
Then, $\tau$ is a Markov stopping time called first {\it hitting time} of $H$.
Note that $\tau_2$ is not generally a hitting time of $H$, that is, $x_{\tau_2}$ does not necessarily belong to
the subset $H$. Although, if, in addition, $H$ is  a subgroup of $G$, then $\tau_n$ is precisely the moment when
the random walk returns to $H$ for the $n$-th time. Hence, in this case, $\mu_{\tau}$ is the probability measure induced on the recurrent subgroup $H$.
Consequently, $(G,\mu_{\tau})$ is the induced random walk on the recurrent subgroup $H$.
\end{ex}
\medskip

By \cite{BK2013}, the Poisson boundary of $(G,\mu_{\tau})$ is the same as the Poisson boundary of $(G,\mu)$.
However, Markov stopping times cannot produce all random walks with the same Poisson boundary as the following trivial
example shows:
\begin{ex}\label{trivial}
 Let $G=\Bbb Z_2=\{0,1\}$, then the Poisson boundary of $(\Bbb Z_2,\mu)$ is trivial for any probability measure $\mu$.
 Consider $\mu=\delta_1$, then $(0,1,0,1,0,\cdots)$ is the only sample path. Hence, in this situation, any Markov stopping time must be
 a constant function. Consequently,  the only probability measures that can be obtained by the above procedure for the random walk $(\Bbb Z_2,\delta_1)$ are convolution powers of $\delta_1$, which are $\delta_0$ and $\delta_1$.
 Therefore, the above procedure  cannot produce the probability measure $\dfrac{1}{2}(\delta_0+\delta_1)$.
\end{ex}
In the next section, we will extend our random walk to a larger space and apply Markov stopping times  to the new space to obtain
 even more random walks with the same Poisson boundary.
\subsection{Randomized Markov stopping times}
Let $(\Omega,\theta)$ be a Lebesgue space. Define transition probabilities
$\pi=\{\pi_{(g,\omega)}\}$ on the probability space $(G\times \Omega,\mu\otimes\theta)$ as follows
$$
\pi_{(g,\omega)}=\sum_h\mu(h)\delta_{gh}\otimes\theta.
$$

Denote by $\p_{\delta_e\otimes m}$ the  probability measure on the product space of the extended chain with
the initial distribution $\delta_e\otimes m$.


In this construction, we replace a sample path $\x=(x_n)$ with its ``randomization''
$$(\x,\boldsymbol\omega)=((e,\omega_0),(x_1,\omega_1),\cdots),$$
where $\omega_i$'s are independent and identically $\theta$-distributed random variables.

In particular, the distribution of this Markov chain at time $n$ is $\pi^n_{(g,\omega)}=\sum\limits_h\mu^{*n}(h)\delta_{gh}\otimes\theta.$

\medskip

Now, let $\tau:G^{\Bbb N}\times \Omega^{\Bbb N}\to \Bbb N$ be a Markov stopping time for the extended Markov chain.
Then, we shall call $\tau$  a {\it randomized Markov stopping time} of the random walk $(G,\mu)$.
Note that, if $m$ is concentrated on one point, then any randomized Markov stopping time is just a usual Markov stopping time.

The transformation $U$ (see equation \ref{induced}) in the path space of the original random walk naturally extends to the following transformation of
the path space of the randomized random walk:
$$
\mathcal{U}(\x,\boldsymbol{\omega})=(x_1^{-1}x_{n+1},\omega_{n+1})_{n\geq0},
$$
and then construct $\tau_n$ on the extended Markov chain as in the previous section.
By using $\mathcal{U}$ one can define the iterated stopping times $\tau_n$  for the randomized random walk in the same way as in formula (\ref{iteration}) above.

By projecting  $(x_{\tau_n},\omega_{\tau_n})_{n\geq0}$ onto the first component,  $(x_{\tau_n})_{n\geq0}$,
we obtain a new random walk $(G,\mu_{\tau})$ with
$$
\mu_{\tau}(g)=\p_{\delta_e\otimes m}\{(\x,\boldsymbol{\omega})\ :\ \x_{\tau}=g\}.
$$

\begin{thm}\cite{BK2013}\label{BK}
Let $\tau$ be a randomized Markov stopping time on the random walk $(G,\mu)$, then $\Gamma(G,\mu)=\Gamma(G,\mu_{\tau})$.
\end{thm}
This randomization allows one to construct  even more random walks such that their Poisson boundaries are the same
as for the original random walk. For example, now we can obtain convex combinations of convolution powers (cf. Example~\ref{trivial}).

\begin{ex}\label{convex}\cite{BK2013}
Let $\Omega=\Bbb N$ with probability measure $\theta$.
Define a Markov stopping time $\tau$ as
$$
\tau((h_1,\omega_1),(h_2,\omega_2),\cdots)=\omega_1.
$$
Then
$
\mu_{\tau}=\sum\limits_{n\geq0}\theta(n)\mu^{*n}.
$
\end{ex}

\begin{ex}\label{wil2}\cite{BK2013}
Let  $\alpha$ and $\beta$ be two sub-probability measures with  $\mu=\alpha+\beta$. Then there is a
randomized Markov stopping time $\tau$ such that
$$
\mu_{\tau}=\beta+\sum_{i\geq1}\alpha^{*i}*\beta.
$$
\end{ex}
This transformation, $\mu=\alpha+\beta\mapsto\mu'=\beta+\sum\limits_{i\geq1}\alpha^{*i}*\beta$,  was introduced by Willis \cite{W90}. He also proved that $\Gamma(G,\mu)=\Gamma(G,\mu')$ .
If, in addition, $\alpha$ and $\beta$ are  mutually singular, then  this example reduces to Example~\ref{wil}.
\section{Entropy of transformed measures}  The aim of this  section is to study the relation between the asymptotic entropies of
 the transformed random walks defined as above and of the original random walk .
\subsection{Asymptotic entropy of a random walk}
The quantity
$$
H(\mu)=-\h
$$
is called the {\it entropy} of $\mu$.
The sequence $\{H_n\}$ is sub-additive, i.e., $H_{n+m}\leq H_n+H_m,$
where $H_n=H(\mu^{*n})$.
Therefore the limit of $H_n/n$  exists. The {\it asymptotic entropy} of $(G,\mu)$ is defined  as
\begin{equation}\label{def asy}
h(\mu)=\lim\limits_n\dfrac{H_n}{n}.
\end{equation}
{\bf Throughout the rest of the paper,  we always assume that $H(\mu)=H_1$ is finite},
in which case $h({\mu})\leq H_1<\infty$.

\medskip

The asymptotic entropy can  also be obtained by Shannon's formula:
\begin{thm}\cite{KV83, De80}\label{shannon}
For $\p$-almost every sample path $(x_n)$,
$$
-h({\mu})=\lim\dfrac{1}{n}\log{\mu^{*n}(x_n)}.
$$
Moreover, the convergence holds in $L^1(\p)$.
\end{thm}
Shannon's formula leads to the following description of the asymptotic entropy of $(G,\mu)$ as well:
\begin{thm}\cite{KV83}\label{harmonic}
For $\p$-almost every  sample path $\x=(x_n)$,
$$
h({\mu})=\lim_n\dfrac{1}{n}\log \dfrac{dx_n\nu}{d\nu}\bnd(\x).
$$

\end{thm}

\subsection{Main results}
As a motivation, let us consider the asymptotic entropy of
the random walk determined by the $k$-fold convolution of $\mu$, i.e., $(G,\mu^{*k})$. By definition of the asymptotic entropy
\begin{equation}\label{motivation}
h(\mu^{*k})=\lim_n\dfrac{1}{n}H_{kn}=\lim_nk\dfrac{H_{kn}}{kn}=kh({\mu}),
\end{equation}
which can be rewritten as
\begin{equation}\label{expectation}
h(\mu_{\tau})=\e(\tau)h({\mu}),
\end{equation}
where $\tau$ is the constant Markov stopping time $k$ (see Example~\ref{convolution}).

 \medskip

 The aim of the next theorem is to show  that  equality (\ref{expectation}) can be generalized to  all Markov stopping times (and to all randomized Markov stopping times as we shall do later, see Theorem~\ref{markov extension}) with finite expectation. This is analogous to Abramov's theorem for the entropy of induced dynamical systems:
 \begin{thm}\label{abramov}\cite{Ab59}
 Let $(X,\phi,\mu)$ be an ergodic measure preserving dynamical system. If $A$ is a  measurable subset of $X$ with $\mu(A)>0$, then
 $$
 h({\mu_A},\phi_A)=\dfrac{1}{\mu(A)}h({\mu},\phi).
 $$
 \end{thm}
 In order to complete the analogy, note that the constant $\dfrac{1}{\mu(A)}$ is equal to the expectation
 of the return time to the set $A$ (Kac formula) \cite{Kac47}.
\medskip

 Since the sequence $\{\mu_{\tau_n}\}$  is not, generally speaking, a subsequence of the sequence of convolution powers of $\mu$, the generalization of the equality (\ref{expectation}) cannot be done by the same trick as in (\ref{motivation}).

\begin{thm}\label{main}
Let $\tau$ be a Markov stopping time with a finite expectation $\e(\tau)$. Then
$H(\mu_{\tau})$ is also finite, and
$$
h({\mu_{\tau}})=\e(\tau)h({\mu}).
$$
\end{thm}

The proof is based on the fact that each sample path of  the random walk transformed by a Markov stopping time is a subsequence of
 the corresponding sample path of the original random walk, and  the description of the asymptotic entropy as the exponential
  growth rate of the Radon--Nikodym derivatives  of the translates of the harmonic measure
   along the sample path (Theorem~\ref{harmonic}). In order to  apply Theorem~\ref{harmonic}, we need  the Poisson boundaries of the random walks $(G,\mu)$ and $(G,\mu_{\tau})$ to be the same
(which follows from Theorem~\ref{BK}), and  the entropy of the random walk $(G,\mu_{\tau})$ to be finite.

\medskip

First, we  show that the entropy of $\mu_{\tau}$ is finite.

\medskip
\begin{lem}\label{finite}
If $\e(\tau)<\infty$, then the entropy of $\mu_{\tau}$ is also finite, and
$$
H(\mu_{\tau})\leq\e(\tau)H_1.
$$

\end{lem}
\begin{proof}
Let
$$
M_n(\x)=nH_1+\log{\mu^{*n}(x_n)}.
$$
Since $H_1$ is finite, $M_n$'s are integrable.
If $\mathcal{A}_0^n$ is the $\sigma$-algebra generated by the position of the random walk $(G,\mu)$ between time $0$ and $n$, then
$$
\e(M_{n+1}|\mathcal{A}_0^n)(\x)=(n+1)H_1+\sum_h\mu(h)\log{\mu^{*(n+1)}(x_nh)}.
$$
The sequence $\{M_n,\mathcal{A}_0^n\}$ is a sub-martingale, i.e.,
$$
\e(M_{n+1}|\mathcal{A}_0^n)(\x)\geq(n+1)H_1+\sum_h\mu(h)(\log\mu^{*n}(x_n)+\log{\mu(h))}= M_n(\x),
$$
because
\begin{equation}\label{equality}
\mu^{*(n+1)}(x_nh)\geq\mu^{*n}(x_n)\mu(h).
\end{equation}
Let $\tau\wedge n=\min\{\tau,n\}$. Then, Doob's optional theorem (see \cite[p.~300]{Doob53}) implies that
$$
\e(M_{\tau\wedge n})~\geq\e(M_0)=0,
$$
and consequently,
$$
\e(\tau\wedge n)H_1+\sum_g\p(x_{\tau\wedge n}=g)\log{\p (x_{\tau\wedge n}=g)}\geq0.
$$
Since $\e(\tau\wedge n)\leq \e(\tau)$, we can write
\begin{equation}\label{fatu}
-\sum_g\p(x_{\tau\wedge n}=g)\log{\p (x_{\tau\wedge n}=g)}\leq \e(\tau)H_1.
\end{equation}
Applying Fatou's lemma to  inequality (\ref{fatu})  gives
\begin{equation}\label{fatu2}
-\sum_g\liminf_n\p(x_{\tau\wedge n}=g)\log{\p (x_{\tau\wedge n}=g)}\leq \e(\tau)H_1.
\end{equation}
On the other hand, the fact that $\lim\limits_n\p (x_{\tau\wedge n}=g)=\mu_{\tau}(g)$ in combination with the continuity of the function $x\log{x}$ implies that
\begin{equation}\label{fatu3}
\lim_n\p(x_{\tau\wedge n}=g)
\log{\p(x_{\tau\wedge n}=g)}=\mu_{\tau}(g)\log{\mu_{\tau}(g)}.
\end{equation}
Now, by (\ref{fatu2}) and (\ref{fatu3}) we obtain
$$
H(\mu_{\tau})=-\sum_g\lim_n\p(x_{\tau\wedge n}=g)\log{\p (x_{\tau\wedge n}=g)}\leq \e(\tau)H_1.
$$
Therefore, we have proved that
$$
\lim_{n\to\infty}H(\mu_{\tau\wedge n})=H(\mu_{\tau}).
$$
\end{proof}

\subsection{Proof of Theorem~\ref{main}}
Now, we can find the asymptotic entropy of the random walk $(G,\mu_{\tau})$. Since the expectation
of the Markov stopping time $\tau$ is finite,   Lemma~\ref{finite} implies that $H(\mu_{\tau})$ is also finite. Therefore, Theorem~\ref{harmonic} implies that
$$
h({\mu_{\tau}})=\lim_n\dfrac{1}{n}\log \dfrac{dx_{\tau_n}\nu}{d\nu}\bnd(\x)
$$
for almost every sample path $\x=(x_n)$.
Since $U$ is a measure preserving transformation,
$\e(\tau_n)=n\e(\tau)$, and, moreover,
$\lim\limits_n\dfrac{\tau_{n}(\x)}{n}=\e(\tau)$ for almost every sample path $\x$. Therefore, it is obvious that
$$
h(\mu_{\tau})=\lim_n\dfrac{\tau_{n}(\x)}{n}\dfrac{1}{\tau_{n}(\x)}\log \dfrac{dx_{\tau_{n}}\nu}{d\nu}\bnd(\x).
$$
By applying Theorem~\ref{harmonic} to the random walk $(G,\mu)$, we will have
$$
h(\mu_{\tau})=\e(\tau)h(\mu).
$$

\subsection{Entropy of random walks transformed  via a randomized Markov stopping time}
Let $\tau$ be a randomized Markov stopping time with finite expectation.
If we replace the sub-martingale in Lemma~\ref{finite} with
 $$
 M_n(\x,\boldsymbol{\omega})=nH_1+\log\mu^{*n}(x_n),
 $$
 then finiteness of the entropy of $\mu_{\tau}$ can be obtained by reproducing the proof of Lemma~\ref{finite}.
Since the Poisson boundary of $(G,\mu_{\tau})$ is the same as the Poisson boundary $(G,\mu)$,  the proof of Theorem~\ref{expectation} applies. Hence, we have

\begin{thm}\label{markov extension}
Let $\tau$ be a randomized Markov stopping time for the random walk $(G,\mu)$. If  $\e(\tau)$ is finite, then
$$
h(\mu_{\tau})=\e(\tau)h({\mu}).
$$
\end{thm}
The combination of Example~\ref{convex} and the preceding theorem implies the following result of Kaimanovich \cite{K83}.
\begin{ex}\label{kaiver}
Let  $\mu'=\sum\limits_{k\geq0}a_k\mu_k$, where $\sum\limits_{k\geq0}a_k=1$ and $a_k\geq0$. Then
$$
h({\mu'})=(\sum_{k\geq0}ka_k)h({\mu}).
$$
\end{ex}

\begin{ex}
Let $\mu=\alpha+\beta$ be as in Example~\ref{wil2}.
Let $\mu'=\beta+\sum\limits_{i\geq1}\alpha^{*i}*\beta$, then
$$
h({\mu'})=\dfrac{1}{\|\beta\|}h(\mu),
$$
where $\|\beta\|$ is the total mass of $\beta$.
\end{ex}

\subsection{$\mu$-boundary}
A {\it $\mu$-boundary} $(\Gamma_{\xi},\nu_{\xi})$ is the quotient of the Poisson boundary with respect to a  $G$-invariant measurable partition $\xi$ (see, \cite{Fu70, K00}).  The {\it differential entropy} of a $\mu$-boundary is defined as
\begin{equation}\label{mubnd}
E_{\mu}(\Gamma_{\xi},\nu_{\xi})=\sum_g\mu(g)\int\log{\dfrac{dg\nu_{\xi}}{d\nu_{\xi}}(g\gamma_{\xi})}d\nu_{\xi}(\gamma_{\xi}).
\end{equation}
Kaimanovich \cite{K83} showed that the asymptotic entropy $h({\mu})$ is the upper bound for
 the asymptotic entropies of $\mu$-boundaries, i.e.,
$$
E_{\mu}(\Gamma_{\xi},\nu_{\xi})\leq h({\mu}).
$$
Moreover, he \cite{K00}  proved an analogue of  Theorem~\ref{harmonic} for $\mu$-boundaries.
Therefore the claim of  Theorem~\ref{markov extension} is also valid the differential entropy of $\mu$-boundaries:

\begin{thm}
Let $\xi$ be a measurable $G$-invariant partition of the Poisson boundary $(G,\mu)$.
If $\tau$ is a randomized Markov stopping time with finite expectation, then
$$
E_{\mu_{\tau}}(\Gamma_{\xi},\nu_{\xi})=\e(\tau)E_{\mu}(\Gamma_{\xi},\nu_{\xi}).
$$

\end{thm}

\section{Rate of escape}
In this section,  we establish a relationship between the escape rates of the  random walk transformed via a randomized Markov stopping time and of the original one.

\medskip

\begin{df}
A gauge $\g=(\g_n)$ on a group $G$ is an increasing sequence of subsets of  $G$ such that $G=\bigcup\limits_n\g_n$. A gauge function $|.|_{\g}$ on $G$ defined as
$$
|g|=|g|_{\g}=\min\{n\ : \ g\in\g_n\}.
$$

A gauge $\g$ is called {\it sub-additive}, whenever its gauge function is sub-additive.
\end{df}
A measure $\mu$ has a finite {\it first moment} with respect to a gauge $\g$, if
$$
|\mu|=\sum_g|g|\mu(g)
$$
is finite.
If  $\g$ is a sub-additive gauge, then obviously
$$
|x_{n+m}|\leq|x_n|+|(U^n\x)_m|,
$$
so that
Kingman's sub-additive theorem implies
\begin{thm}\cite{De80, G79}
Let $\g$ be a  sub-additive gauge and $\mu$ have a finite first moment with respect to $\g$. Then
\begin{equation}
\ell(G,\mu,\g)=\ell(\mu)=\lim_n\dfrac{|x_n|}{n}
\end{equation}
exists for $\p$-almost every sample path $\x=(x_n)$, and also in $L^1(\p)$.
\end{thm}
The quantity $\ell(\mu)$ is called the rate of escape (drift) of the random walk $(G,\mu)$ with respect to the gauge $\g$.
\medskip
\begin{ex}
Let $G$ be a group generated by  a finite set $S=S^{-1}$. Then, $(S^n)_n\geq1$ is a sub-additive gauge.
\end{ex}
\begin{ex}
Let $(G,\mu)$  be a {\it transient} random walk. Denote by $F(g)<1$ the probability that
the random walk $(G,\mu)$ ever visits a point $g\in G$.
Let $\g_n=\{g\in G : -\ln{F(g)}\leq n\}$, then $\g=(\g_n)_{n\geq0}$ is a sub-additive gauge.
Actually,  $-\ln{F(g)}$ is the ``distance'' of $g$ from  the identity element in the so-called {\it Green metric}
\cite{BS07}.
The rate of escape with respect to the Green metric is equal to the asymptotic entropy \cite{BHM08}.
\end{ex}
\begin{thm}
Let $\mathcal{G}$ be a sub-additive gauge for  group $G$ and $\mu$ have a finite first moment.
Let $\tau$ be a randomized Markov stopping time with finite expectation $\e(\tau)$.
Then, the probability measure $\mu_{\tau}$ has a finite first moment with respect  to the gauge $\g$,
and
\begin{equation}\label{rate}
\ell(\mu_{\tau})=\e(\tau)\ell(\mu).
\end{equation}
\end{thm}
\begin{proof}
Let $L_n(\x,\boldsymbol{y})=n\e(|x_1|)-|x_n|$. Then
$$
\e(L_{n+1}(\x,\boldsymbol{y})|\mathcal{A}_0^n)=(n+1)\e(\tau)-\sum_h|x_nh|\mu(h).
$$
The sub-additivity of the gauge $\mathcal{G}$ implies that $\{L_n,\mathcal{A}_0^n\}$ is a sub-martingale.
Now applying the same argument as in Lemma~\ref{finite} and Theorem~\ref{main}, we get the equality (\ref{rate}).
\end{proof}
\bibliographystyle{alpha}


\end{document}